\documentclass[12pt,twoside]{article} 
\usepackage{geometry}
\usepackage[all]{xy}
\usepackage{amsmath,amsthm}
\usepackage{amssymb}
\usepackage{listings}
\usepackage{xcolor}
\usepackage{color}
\usepackage{graphicx}
\usepackage{verbatim}
\usepackage{epstopdf}
\usepackage{subfigure}
\usepackage{algorithm}
\usepackage{algorithmicx}
\usepackage{algpseudocode}
\usepackage[colorlinks,
            linkcolor=blue,
            anchorcolor=blue,
            citecolor=blue]{hyperref}

\date{} 

\begin{document} 



\centerline{}
\centerline{}
\centerline{\Large{\bf Laplace Transforms of Improper Random Variables}} 

\centerline{} 

\centerline{\bf {Janhavi PRABHU}} 

\centerline{} 

\centerline{Department of Mathematics}
 
\centerline{Birla Institute of Technology and Science} 

\centerline{Pilani, Goa Campus, India} 

\centerline{} 

\centerline{\bf {Myron HLYNKA}} 

\centerline{} 

\centerline{Department of Mathematics and Statistics} 

\centerline{University of Windsor} 

\centerline{Windsor, Ontario, Canada N9B 3P4} 

\newtheorem{Theorem}{\quad Theorem}[section] 

\newtheorem{Definition}[Theorem]{\quad Definition} 

\newtheorem{Corollary}[Theorem]{\quad Corollary} 

\newtheorem{Lemma}[Theorem]{\quad Lemma} 

\newtheorem{Example}[Theorem]{\quad Example} 

\centerline{}
{\footnotesize This article is distributed under the Creative Commons by-nc-nd Attribution License. Copyright $\copyright$ 20xx Hikari Ltd.}

\begin{abstract}
The probabilistic interpretation of Laplace transforms is used to help to describe the Laplace Transform $L(s)$ of improper random variables. In particular, busy periods in queueing models are examined. The value of $L(0)$ is explained in certain special cases. 
\end{abstract}



\section{Introduction}
A typical Laplace Transform $L(s)$ for a non-negative continuous proper random variable (X) probability density function ($f(x)$) satisfies $$L_X(s)=\int_0^{\infty}f(x)e^{-sx}dx.$$
Hence $L_X(0)= 1$ for proper pdf's, and $lim_{s\rightarrow \infty}L(s)=0$ and $L(s)$ is a non-increasing function of $s$. 

The probabilistic interpretation of Laplace Transforms was first studied by van Dantzig (\cite{6}). Further explanations appear in Runnenberg (\cite{5}), in 
Kleinrock (\cite{2}), and Roy (\cite{4}). This probabilistic interpretation can be expressed as follows. 
\begin{Theorem} Let $X$ be a continuous random variable with non-negative support having pdf $f(x)$. Let $Y$ be a 
``catastrophe'' random variable independent of $X$, which has an exponential distribution with rate $s$. Then 
$$L_X(s)=P(X<Y).$$
\end{Theorem}
\begin{proof}
	Since $F_Y(y)=1-e^{-ys}=P(Y\leq y)$, we have $e^{-sy}=P(Y>y)$ so  
$$L_X(s)=\int_0^{\infty}f(x)e^{-sx}dx=\int_0^{\infty}f(x)P(Y>x)dx=P(X<Y)$$
\end{proof}

An improper non negative continuous random variable with pdf $f(x)$ is such that $\int_0^{\infty}f(x) dx=1-k$ and
$P(X=\infty)=k$, for some $k\in (0,1)$. See Mynbaev (\cite{3}), for example.

In this paper, we are interested in looking at improper random variables and interpreting the Laplace transform using the probabilistic interpretation. This will be done by considering busy periods in an M/M/1 queueing system. We show the effect of the stability conditions on the appearance of the Laplace transform. Our analysis will also use branching process methods. 

\section{Analysis of Busy Periods of M/M/1 queueing system}

Consider a stable $M/M/1$ queueing system with $\lambda<\mu$ where $\lambda$ is the arrival rate and $\mu$ is the service rate. The Laplace Transform of the busy period is known and we give a probabilistic interpretation below. 

\begin{Theorem}For an $M/M/1$ system with busy period length $B$, 
	$$L_B(s)= \dfrac{\mu+\lambda +s - \sqrt{(\mu+\lambda+s)^2-4\lambda \mu}}{2 \lambda}.$$
	\end{Theorem}
\begin{proof}
Let $L_i(s)$ be the time to empty a queueing system if the current load consists of $i$ customers. A busy period begins with a single customer so $L_B(s)=L_1(s)$.
Let $Y$ be a catastrophe random variable at rate $s$. Now the next event could be an arrival (with probability $\dfrac{\lambda}{\lambda+\mu +s}$) or a completion (with probability $\dfrac{\mu}{\mu+\lambda +s}$) or a catastrophe. If the next event is a service completion, that would end the busy period. If the next event is an arrival, then there would be 2 customers and each of those would have to drop a level in order to complete the busy period. Then\\
\begin{align*}
	L_B(s)&=P(B<Y)=P(\text{next event is service completion})\\
	&\qquad +P(\text{next event is arrival})P(\text{system empties from 2 customers})\\
	&=\dfrac{\mu}{\mu+\lambda +s}+\dfrac{\lambda}{\mu+\lambda +s}L_2(s)
\end{align*}\
 But $L_2(s)=L_1(s)^2$ since we need to drop from 2 customers to 1 customer and then 1 customer to zero customers and these are distributionally the same. 
 Thus 
 $$L_B(s)=\dfrac{\mu}{\mu+\lambda +s}+\dfrac{\lambda}{\mu+\lambda +s}L_B(s)^2$$
 This is a quadratic and solving yields our result (by keeping the root that is in the interval [0,1]). 

	\end{proof}

\begin{Example}
First we look at the graph of $L_B(s)$ when the system is stable. Take $\lambda=3 $ and $\mu=4$. 
Then $L(s)= \dfrac{7+s-\sqrt{7+s)^2-48}}{6}$ 
 See Figure 1.
\begin{figure}[htb]
   \centering
   \includegraphics[scale=0.6]{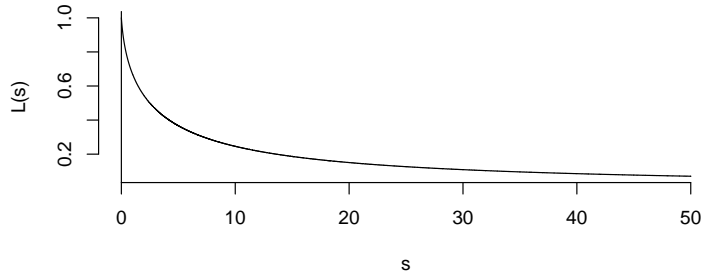}
\caption{LT of busy period (stable case) }
\end{figure}
\end{Example}

\begin{Example}
Next, we look at an unstable $M/M/1$ queueing system with $\lambda=4$ and $\mu=3$ (i,.e. the arrival rate is greater than the service rate). We can still use the LT expression. Here $L(s)=\dfrac{7+s-\sqrt{7+s)^2-48}}{8}$ 
The graph now takes the following form.

 See Figure 2.
\begin{figure}[htb]
	\centering
	\includegraphics[scale=0.6]{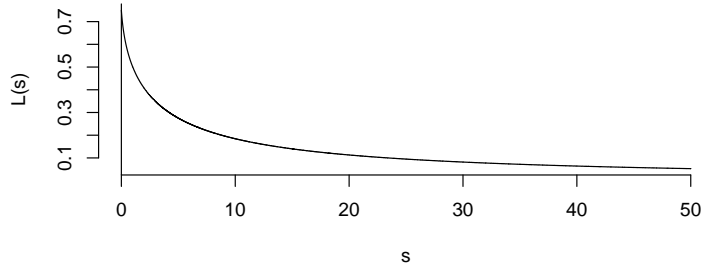}
	\caption{LT of busy period (unstable case) }
\end{figure}
\end{Example}
 In our example, we see that $L(0)=.75$. What does this mean? 
  Fortunately, our probabilistic interpretation of $L_B(s)$ gives us this information.

\begin{Theorem}
	If $L(0)=k$, $k<1$ for an improper non-negative random variable representing a busy period $B$ in a queueing system, then $k$ represents the probability that the busy period will end before the system count enters a path which never returns to 0.  
	\end{Theorem}
\begin{proof}
	Since $B$ is an improper random variable, it is possible that $B$ will take the value $\infty$ (i.e. busy period will never end). Also $k=lim_{s\rightarrow 0^+}L(s) =P(B<Y)$ where $Y$ is exponential at rate $0^+$. So $k$ represents the probability that the busy period will end (and be finite) at least once before the system number increases and never returns to 0. 
		\end{proof}

\begin{Theorem} For an $M/M/1$ queueing system, with arrival rate $\lambda$ and service rate $\mu$, satisfying $\lambda>\mu$, the probability that a busy period ends before the system count never returns to zero is $\dfrac{\mu}{\lambda}$
\end{Theorem}
\begin{proof}
For a discussion of branching processes, see Grinstead and Snell (\cite{1}).
For each customer, define the next generation to consist of 2 items with probability $\beta_2=\dfrac{\lambda}{\lambda + \mu}$ (an arrival plus the original item) and 0 items with probability $\beta_0=\dfrac{\mu}{\lambda+\mu}$. The probability generating function for the next generation is $\beta(z)=\beta_0+\beta_2 z^2$. Extinction would mean that the system reaches 0 at some time. By the Fundamental Theorem of Branching Processes, the probability of extinction $\alpha$ is the smallest positive real root of $\alpha=\beta(\alpha)$. This is a quadratic in $\alpha$ and we get our result. 
\end{proof}

\begin{Example}
	For an $M/M/1$ queueing system, with $\lambda=4$ and $\mu=3$, we previously obtained a Laplace Transform 
	$L(s)=\dfrac{7+s-\sqrt{7+s)^2-48}}{8}$ and computed $L(0)=6/8=.75$. Our theorem above gives the probability of extinction as $\mu/\lambda=3/4=.75$ so our results match. 
\end{Example}

\section{Conclusions}

In this paper, we have explained the value of $L(0)$ for the Laplace Transform of an improper continuous random variable with positive support. A good example of such a random variable is the length of a busy period in a queueing system with arrival rate exceeding service rate. 

If we had a complex queueing network with retrials and reneging and balking and switching of servers, for example, it may still be possible to find the Laplace Transform in terms of the parameters. In some cases, it may be that $L(0)<1$. Fixing all but one of the parameters, we can adjust the single parameter until $L(0)=1$ holds (if possible). The value of the parameter where this change occurs must be on the boundary of stability of the system. So our study of $L(0)$ gives information about the boundary conditions, which is very important.   

One final comment about improper Laplace non negative random variables is that they will have an infinite expected value since there is a positive probability that the random variable will take the value $\infty$. 
 \centerline{}  
 \centerline{} 
{\bf Acknowledgements.} We acknowledge funding and support  from MITACS Global Internship program.

{\bf Received: Month xx, 20xx}
\end{document}